\documentclass{amsart}

\usepackage{amscd}
\usepackage{amssymb}
\usepackage{soul}
\usepackage[normalem]{ulem}
\usepackage{mathrsfs}

\usepackage[all]{xypic}
\usepackage[dvipsnames] {xcolor}
\usepackage{tikz}

\numberwithin{equation}{section}

\theoremstyle{plain} 
\newtheorem{theorem}[equation]{Theorem}
\newtheorem{lemma}[equation]{Lemma}

\newtheorem{proposition}[equation]{Proposition}
\newtheorem{corollary}[equation]{Corollary}

\theoremstyle{definition}
\newtheorem{definition}[equation]{Definition}

\newtheorem{example}[equation]{Example}

\newtheorem{remark}[equation]{Remark}
\newtheorem*{remark*}{Remark}

\newcommand{\defining}[1]{{\emph{#1}}}
\newcommand{\definedas}{:=}

\newcommand{\reals}{{\mathbb{R}}}
\newcommand{\integers}{{\mathbb{Z}}}

\newcommand{\rationals}{{\mathbb{Q}}}

\newcommand{\field}{{\mathbb{F}}}

\def\doCal#1{%
\ifx#1\doAllCalEnd\def\doAllCal{\relax}\else%
 \expandafter\edef\csname#1cal\endcsname{{\noexpand\mathcal #1}}\fi}
\def\doAllCal#1{\doCal#1\doAllCal}
\doAllCal ABCDEFGHIJHKLMNOPQRSTUVWXYZ\doAllCalEnd

\def\doBar#1{%
\ifx#1\doAllBarEnd\def\doAllBar{\relax}\else%
 \expandafter\edef\csname#1bar\endcsname{{\noexpand\overline{#1}}}\fi}
\def\doAllBar#1{\doBar#1\doAllBar}
\doAllBar ABCDEFGHIJHLMNOPQRSTUVWXYZabcdefghijhlmnopqrstuvwxyz\doAllBarEnd
\def\doWiggle#1{%
\ifx#1\doAllWiggleEnd\def\doAllWiggle{\relax}\else%
 \expandafter\edef\csname#1wiggle\endcsname{{\noexpand\tilde{#1}}}\fi}
\def\doAllWiggle#1{\doWiggle#1\doAllWiggle}
\doAllWiggle ABCDEFGHIJHLMNOPQRSTUVWXYZabcdefghijklmnopqrstuvwxyz\doAllWiggleEnd

\def\doBold#1{%
\ifx#1\doAllBoldEnd\def\doAllBold{\relax}\else%
 \expandafter\edef\csname#1bold\endcsname{{\noexpand\bf #1}}\fi}
\def\doAllBold#1{\doBold#1\doAllBold}
\doAllBold ABCDEFGHIJHKLMNOPQRSTUVWXYZabcdefghijklmnopqrstuvwxyz\doAllBoldEnd

\def\doBbb#1{%
\ifx#1\doAllBbbEnd\def\doAllBbb{\relax}\else%
 \expandafter\edef\csname#1bb\endcsname{{\noexpand\mathbb{#1}}}\fi}
\def\doAllBbb#1{\doBbb#1\doAllBbb}
\doAllBbb ABCDEFGHIJHKLMNOPQRSTUVWXYZ\doAllBbbEnd

\newcommand{\Hugestrut}{\mbox{{\Huge\strut}}}

\newcommand{\Largestrut}{\mbox{{\Large\strut}}}

\makeatletter
\newcommand{\switchmargin}{
\if@reversemargin
\normalmarginpar
\else
\reversemarginpar
\fi
}
\makeatother

\definecolor{llgray}{RGB}{230,230,230}
\definecolor{lblue}{RGB}{217,249,237}
\newcommand{\highlight}[1]{\ifmmode{\text{\sethlcolor{llgray}\hl{$#1$}}}\else{\sethlcolor{llgray}\hl{#1}}\fi}
\newcommand{\highlighteva}[1]{\ifmmode{\text{\sethlcolor{llgray}\hl{$#1$}}}\else{\sethlcolor{lblue}\hl{#1}}\fi}

\DeclareMathOperator{\hocolim}{hocolim}
\DeclareMathOperator{\Hom}{Hom}

\DeclareMathOperator{\Out}{Out}

\newcommand{\pdash}{$p$\kern1.3pt-}

\newcommand{\Zpinfinity}{\integers/{p}^{\infty}}
\newcommand{\Zqinfinity}{\integers/{q}^{\infty}}
\newcommand{\Ztwoinfinity}{\integers/{2}^{\infty}}

\newcommand{\subgroupeq}{\subseteq}

\newcommand{\subgroup}{\subset}

\newcommand{\sd}{{\bar{s}d}}

\newcommand{\xlongrightarrow}[1]{\xrightarrow{\ #1\ }}

\newcommand\restr[2]{{
  \left.\kern-\nulldelimiterspace 
  #1 
  \vphantom{\big|} 
  \right|_{#2} 
  }}

\newcommand{\suchthat}[1]{\left|\, #1 \right. }

\newcommand{\pcomplete}[1]{{#1}_{p}^{\wedge}}

\renewcommand{\bar}{\widebar} 

\usepackage{hyperref}
\hypersetup{colorlinks=true,linkcolor=violet,citecolor=purple} 

\usepackage{marginnote}

\usepackage[nolabel]{showlabels} 
\definecolor{llgray}{RGB}{230,230,230}
\newcommand{\Top}{\operatorname{Top}}
\newcommand{\widebar}[1]{{\overline{#1}}}

\newcommand{\longleftrightarrows}{\xymatrix@1@C=16pt{
\ar@<0.4ex>[r] & \ar@<0.4ex>[l]
}}

\renewcommand{\atop}[1]{{\let\scriptstyle\textstyle\let\scriptscriptstyle\scriptstyle\substack{#1}}}

\newcommand{\fusion}{$\Fcal$}
\newcommand{\group}{K}
\newcommand{\newgroup}{H}

\newtheorem*{maintheorem}{Theorem~\ref{theorem: comparison of chains}}
 
\newcommand{\maintheoremtext}
{Let $\Pbold_0\subgroupeq \ldots\subgroupeq \Pbold_{k}$ be a chain of \pdash toral subgroups of a compact Lie group~$G$, and let 
$P_0\subgroupeq \ldots\subgroupeq P_{k}$ be a chain of discrete \pdash toral subgroups such that each $P_i$ is a \pdash discretization of~$\Pbold_i$.
Then 
\[
N_{G}\left(P_0\subgroupeq\ldots\subgroupeq P_{k}\right)
\longrightarrow 
N_{G}\left(\Pbold_0\subgroupeq \ldots\subgroupeq \Pbold_{k}\right)
\]
induces a mod~$p$ equivalence of classifying spaces. }

\newtheorem*{chainstheorem}{Theorem~\ref{theorem: identification of chains}}

\newcommand{\chainstheoremtext}{
Let $\Sbold$ be a maximal \pdash toral subgroup of a compact Lie group~$G$, with \pdash discretization $S\subgroupeq\Sbold$. The closure map $P\mapsto\Pbar$ defines an injective map
\begin{equation*}
\begin{gathered}
\xymatrix{
\left\{\strut P_0\subgroupeq \ldots \subgroupeq P_k\subgroupeq S
\suchthat{
      \textrm{all $P_i$ are \fusion-centric and \fusion-radical}}\right\}
      /G\ar[d]\\
\left\{\strut \Pbold_0\subgroupeq \ldots \subgroupeq\Pbold_k \subgroupeq\Sbold
\suchthat{
       \textrm{all $\Pbold_i$ are \pdash toral,
       \pdash centric, and \pdash stubborn}}\right\}
       /G.
}
\end{gathered}
\end{equation*}
The map is a one-to-one correspondence if $\pi_0G$ is a \pdash group.}

\begin{document}

\title{Normalizers of chains of discrete \pdash toral subgroups in compact Lie groups }

\begin{abstract}
In this paper we study the normalizer decomposition of a compact Lie group $G$ using
the information of the fusion system $\Fcal$ of $G$ on a maximal discrete \pdash toral subgroup.
We prove that there is an injective map from the set of conjugacy classes of chains of \fusion-centric, \fusion-radical discrete \pdash toral subgroups to the set of conjugacy classes of chains of 
\pdash centric, \pdash stubborn continuous \pdash toral subgroups. The map is a bijection when $\pi_0(G)$ is a finite \pdash group.
We also prove that the classifying space of the normalizer of a chain of discrete \pdash toral subgroups of $G$ is mod $p$ equivalent to the classifying space of the normalizer of the corresponding  chain of \pdash toral subgroups.
\end{abstract}

\author{Eva Belmont}
\address{Department of Mathematics, University of California San Diego, La Jolla, CA, USA}
\email{ebelmont@ucsd.edu}

\author{Nat\`alia Castellana}
\address{Departament de Matem\`atiques, Universitat Aut\`onoma de Barcelona, and Centre de Recerca Matemàtica, Barcelona, Spain}
\email{natalia@mat.uab.cat}

\author{Jelena Grbi\'{c}}
\address{School of Mathematical Sciences, University of Southampton, Southampton, UK }
\email{j.grbic@soton.ac.uk}

\author{Kathryn Lesh}
\address{Department of Mathematics, Union College, Schenectady NY, USA}
\email{leshk@union.edu}

\author{Michelle Strumila}
\address{School of Mathematics and Statistics,
      Melbourne University, Parkville VIC, Australia}
\email{mstrumila@student.unimelb.edu.au}


\maketitle


\markboth{\sc{Belmont, Castellana, Grbi\'{c}, Lesh, and Strumila}}
{\sc{Normalizer decompositions}}


\section{Introduction}

In \cite{Dwyer-Homology}, Dwyer formalized and unified three homology decompositions for the \pdash completed classifying space of 
a finite group~$G$ based on a collection of \pdash subgroups: the centralizer decomposition, the subgroup decomposition, and the normalizer decomposition. The first two had been studied in \cite{JM} and \cite{JMO} for compact Lie groups, and the normalizer decomposition was new in this context. 
Dwyer showed that for a given collection of subgroups of a finite group~$G$, either all three decompositions give the correct homotopy type for $\pcomplete{BG}$ or none of them do. Such decompositions in the setting of Lie groups have since been studied by other authors, for example \cite{CLN,Libman-Minami, slominska-webb-conj,strounine}.
In particular, in \cite{Libman-Minami} Libman 
gives a normalizer decomposition and then 
unifies the three homology decompositions for Lie groups, as Dwyer did in the finite group case. 

Recently, the homotopy theory of \pdash local compact groups
\cite[Defn~4.2]{BLO-Discrete}
has provided a new, more general framework for dealing with the homotopy type of \pdash completed classifying spaces of compact Lie groups, 
in addition to other examples coming from finite loop spaces (\cite{BLO-LoopSpaces}). 
One works with discrete \pdash toral groups
(Definition~\ref{definition: terms}) instead of \pdash toral groups. 
The formal structure of a \pdash local compact group consists of a triple $(S,\Fcal,\Lcal)$ where $\Fcal$ is a saturated fusion system over the discrete \pdash toral groups $S$ and $\Lcal$ is a centric linking system associated to~$\Fcal$. But in view of \cite[Thm~B]{LL-ExistenceL} a \pdash local compact group is equivalent to just a pair $(S,\Fcal)$, namely a saturated fusion system over a discrete \pdash toral group.

When the \pdash local compact group 
arises from a compact Lie group~$G$, it encodes 
the essential \pdash local information needed to uniquely determine the homotopy type of~$\pcomplete{BG}$ (see \cite{BLO-Discrete}, \cite{Chermak}, \cite{Oliver-ExistenceL}, \cite{LL-ExistenceL}).
A~great advantage of studying Lie groups via this theory is that it reduces the study of a topological group to the study of a collection of \emph{discrete} subgroups.
There are also other interesting examples
of \pdash local compact groups. 
For example, one can construct a \pdash local compact group capturing the homotopy type of a \pdash compact group 
(an $\field_p$-finite loop space together with a chosen \pdash complete delooping, see \cite{dwyer-wilkerson-fixed-point}).
Other examples
of \pdash local compact groups 
are given in \cite{BLO-LoopSpaces} and \cite{Gonzalez-Lozano-Ruiz}.

To state the form of a normalizer decomposition more precisely,
consider a collection $\Ccal$ of closed subgroups of a Lie group~$G$.
Define $\sd(\Ccal)$ to be the poset of $G$-conjugacy classes of chains of proper inclusions in~$\Ccal$, say $\Hbold_*\definedas(\Hbold_0\subgroup \cdots \subgroup \Hbold_k)$. 
One can construct a functor $\delta\colon \sd(\Ccal)\rightarrow \Top$, and a natural transformation from $\delta$ to the constant functor with value $BG$,
to induce a map
\begin{equation} \label{eq: normalizer equivalance}
\pcomplete{\left(\underset{\sd(\Ccal)}{\hocolim}\, \delta\right)\!} 
\to \pcomplete{BG}
\end{equation}
such that $\delta(\Hbold_*)\simeq BN_G(\Hbold_0\subgroup \cdots \subgroup \Hbold_k)\definedas B\big(\bigcap_i N_G(\Hbold_i)\big)$. 
The following statement collects results of
\cite[Thm~C,~D]{Libman-Minami}, \cite[Thm~1.4]{JMO}, and \cite[Lemma~9.7]{BLO-Discrete} that establish collections for which the normalizer decomposition correctly computes the \pdash completed homotopy type of~$BG$. 

\begin{theorem}\label{thm:libman-intro}
Let $G$ be a compact Lie group and let $\Ccal$ be either \it{(i)} the collection of nontrivial \pdash radical \pdash toral subgroups or \it{(ii)} the collection of \pdash stubborn \pdash toral subgroups or \it{(iii)} the collection of \pdash centric \pdash toral subgroups of~$G$ (see Definition~\ref{defn: centric, radical, stubborn}). Then \eqref{eq: normalizer equivalance} is an equivalence.
\end{theorem}


Our program's goal, taken up in a forthcoming work \cite{BCGLS-Main}, is a computable setup that generalizes the normalizer decomposition \eqref{eq: normalizer equivalance} from compact Lie groups to \pdash local compact groups. The formalism is a straightforward generalization of the earlier work of Libman \cite{Libman-normalizer} giving a normalizer decomposition for \pdash local {\emph{finite}} groups. In a result similar to Theorem~\ref{thm:libman-intro}, \cite{BCGLS-Main} will also show that if the \pdash local compact group corresponds to the fusion system~$\Fcal$, then 
the full subcategory of~$\Fcal$ consisting 
of $\Fcal$-centric and $\Fcal$-radical subgroups (Definition~\ref{defn: centric, radical, stubborn})
is 
sufficient to determine the homotopy type of the \pdash completed classifying space.  
This result is in the literature for finite groups (\cite[Thm 1.5]{Grodal}) and $p$-local finite groups (\cite[Thm~3.5]{BCGLO}), but not for \pdash local compact groups. 

When it comes to actual computations, the analysis of the spaces coming into our normalizer decompositions for \pdash local compact groups can be delicate. This paper is largely in service of understanding the spaces in the decompositions that we obtain in certain examples. In particular, we need to understand what happens in the case of a \pdash local compact group that arises from a compact Lie group, because we want to compare the decomposition we obtain in \cite{BCGLS-Main} with the earlier one of Libman for the corresponding Lie group \cite{Libman-Minami}.

We turn to a description of the contents of this paper and how they fit into our program. Let $G$ be a compact Lie group, and let 
$\Sbold\subgroupeq G$ be a maximal \pdash toral subgroup of $G$ with maximal discrete \pdash toral subgroup $S\subgroupeq \Sbold$. The corresponding fusion system $\Fcal$ associated to~$G$ is the category whose objects are the discrete \pdash toral subgroups of~$S$, and whose morphisms are given by homomorphisms induced by conjugation by elements of~$G$.
The goal of this paper is to establish that the left side of \eqref{eq: normalizer equivalance}, which is described in terms of chains of continuous \pdash toral groups and the action of~$G$, can instead be described in terms of~$\Fcal$, i.e. in terms of chains of discrete \pdash toral groups of~$G$  and morphisms in~$\Fcal$. 
There are two issues: the indexing category, and the values of the functor~$\delta$. 

Our first theorem addresses the indexing category, by relating conjugacy classes of chains of 
discrete \pdash toral subgroups of a compact Lie group~$G$ to conjugacy classes of chains of continuous \pdash toral subgroups of~$G$.
The following theorem establishes that the desired classes of chains can all be found by considering the \pdash stubborn \pdash toral subgroups of~$G$, which are classified in \cite{Oliver-p-stubborn} for classical groups. 
(See Definition~\ref{defn: p-discretization} for \pdash discretization.) 

\begin{chainstheorem}
\chainstheoremtext
\end{chainstheorem}

Our second theorem deals with the values of the functor $\delta$ in~\eqref{eq: normalizer equivalance}. In particular, we 
relate the mod~$p$ homotopy type of the classifying spaces of normalizers of chains of discrete \pdash toral subgroups to those of chains of continuous \pdash toral subgroups. Since our decomposition for \pdash local compact groups will involve the former, this theorem will relate \emph{(i)} the decomposition given by our \pdash local compact group methods
applied to the case of a compact Lie group and \emph{(ii)} the decomposition for a compact Lie group that is obtained by \cite{Libman-Minami}. 

\begin{maintheorem}
\maintheoremtext
\end{maintheorem}
\noindent The proof introduces the outer automorphism group of a chain $H_0\subgroupeq\cdots \subgroupeq H_k$ (Definition~\ref{defn: outer of a chain}), which turns out to be a finite group and plays an important role in the argument. (See Proposition~\ref{proposition: iso of outer autos discrete to cont}, Lemma~\ref{lemma: out of chain is finite}, and diagram~\eqref{eq: ladder for normalizers}.)

In summary, this paper provides the technical results necessary to compare 
two normalizer decompositions for classifying spaces of compact Lie groups: the one obtained by applying \cite{BCGLS-Main} to a \pdash local compact group arising from a Lie group~$G$, and the earlier one due to Libman~\cite{Libman-Minami}, obtained by techniques using $G$-actions. 
The two decompositions are related by taking closures of discrete \pdash toral subgroups, which brings up surprisingly subtle issues. Hence we develop some useful tools for studying the relationship between discrete \pdash toral groups and their closures, as well as the relationship between the classifying spaces of their respective normalizers.


\subsection*{Notation.}
Throughout the paper, $G$ denotes a compact Lie group. Our convention for conjugation is that $c_g(x)=g^{-1}xg$. We generally use a boldface font to denote a topological group, as opposed to a discrete group, with the exception of the ambient Lie group~$G$ itself. For example, we use $\Pbold$ for a \pdash toral group, and $P$ for a discrete \pdash toral group. 

\subsection*{Organization.}
Section~\ref{section: p-toral} includes background material on \pdash toral and discrete \pdash toral subgroups of a compact Lie group.  Section~\ref{section: relative dptas} presents a key technical result
on \pdash discretization of pairs, which allows us to understand how chains of discrete \pdash toral subgroups conjugate inside their closures. Some of the results of this section already appear in \cite{BLO-Discrete}, but we present some simplified proofs. Section~\ref{section: what are the chains} contains the proof of Theorem~\ref{theorem: identification of chains}.
Lastly, in Section~\ref{section: normalizers} we introduce the group of outer automorphisms of a chain and we prove Theorem~\ref{theorem: comparison of chains}.
 
\subsection*{Acknowledgements.}
This paper is the first part of the authors' Women in Topology III project. A second part of that project will appear in a separate article \cite{BCGLS-Main}. We thank the organizers of the Women in Topology III workshop, as well as the Hausdorff Research Institute for Mathematics, where the workshop was held. 
The Women in Topology III workshop was supported by NSF grant DMS-1901795, the AWM ADVANCE grant NSF HRD-1500481, and Foundation Compositio Mathematica. The second author was partially supported by Spanish State Research Agency through the FEDER-MEC grant MTM2016-80439-P, and the Severo Ochoa and María de Maeztu Program for Centers and Units of Excellence in R$\&$D (CEX2020-001084-M). Finally, the authors gratefully acknowledge exchanges with Bill Dwyer, who suggested the approach to \pdash discretizations used in 
Section~\ref{section: relative dptas}, and also thank the anonymous referee for a careful reading and suggestions for improvements. 


\medskip
\section{\pdash toral and discrete \pdash toral subgroups of Lie groups}
\label{section: p-toral}

In this section, we give background material on \pdash toral and 
discrete \pdash toral subgroups of a compact Lie group~$G$. First, the definitions.

\begin{definition}  \label{definition: terms}
\hfill 
\begin{enumerate}
    \item A group is \defining{\pdash toral of rank~$r$} if it is an extension of a torus of rank~$r$ by a finite \pdash group. 
    \item A \defining{discrete \pdash torus of rank~$r$} is a group isomorphic to a product $\left(\Zpinfinity\right)^r$. 
    \item A \defining{discrete \pdash toral group of rank~$r$} is an extension of a discrete \pdash torus of rank~$r$ by a finite \pdash group.  
\end{enumerate}
\end{definition}

\begin{remark*}
When we want to emphasize the difference between a \pdash toral group and a discrete \pdash toral group, we will sometimes refer to the former as a \emph{continuous \pdash toral group}.
\end{remark*}

The \pdash toral subgroups of a compact Lie group $G$ play a key role in the analysis of the mod~$p$ homology of the classifying space of~$G$, analogous to the role played by the \pdash subgroups in the case of a finite group. However, there is a key difference between the finite and topological contexts: 
subgroups of finite \pdash groups are finite \pdash groups, but subgroups of \pdash toral groups need not be \pdash toral. 
For example, $S^1$ is a \pdash toral group, but it has finite subgroups of order prime to~$p$ (certainly not \pdash toral) as well as the subgroup $\Zpinfinity\subgroup S^1$ (also not \pdash toral, for a different reason).  
By contrast, a subgroup of a discrete \pdash toral group is necessarily another discrete \pdash toral group.
This feature of discrete \pdash toral subgroups of a compact Lie group~$G$ gives them an advantage over continuous \pdash toral subgroups as tools to approximate~$G$. 

A compact Lie group~$G$ admits both maximal 
continuous \pdash toral subgroups and maximal discrete \pdash toral subgroups, both of which have properties analogous to those of the 
Sylow \pdash subgroups of a finite group. 

\begin{proposition}[{\cite[Prop.~9.3]{BLO-Discrete}}] 
\label{proposition: existence of maximal}
Let $G$ be a compact Lie group. 
\begin{enumerate}
    \item Every \pdash toral subgroup (respectively, discrete \pdash toral subgroup) of~$G$ is contained in a maximal one. 
    \item All maximal \pdash toral subgroups (respectively, discrete \pdash toral subgroup) are conjugate in~$G$. 
\end{enumerate}
\end{proposition}

Unfortunately, discrete \pdash toral subgroups will not be good approximations to continuous \pdash toral groups when their group theoretic properties interact badly with their embeddings. In particular, 
a discrete \pdash toral group can have smaller rank than its closure, and this can occur even when the groups are
just tori. For example, $\Zpinfinity$ can be embedded via a homomorphism as a dense subgroup of $S^{1}\times S^{1}$. 
To set apart good approximations from bad ones, we have the following definition. 

\begin{definition}[{\cite[Defn.~9.1]{BLO-Discrete}}]
\label{definition:snuggly-embedded}
A discrete \pdash toral subgroup $P\subgroupeq G$ is \defining{snugly embedded} if $P$ is a maximal discrete \pdash toral subgroup of~$\Pbar$. 
\end{definition}

\begin{lemma}[{\cite[Prop.~9.2]{BLO-Discrete}}]
\label{lem:snuglypcompleted}
If $P\subgroupeq G$ is a snugly embedded discrete \pdash toral group, then
$P\hookrightarrow\Pbar$ induces a homotopy equivalence 
$\pcomplete{(BP)}\simeq \pcomplete{\left(B\Pbar\right)}$.
\end{lemma}

Not all discrete \pdash toral subgroups of a compact Lie group $G$ are snugly embedded. However, since any \pdash toral group possesses maximal discrete \pdash toral subgroups by Proposition~\ref{proposition: existence of maximal}, a \pdash toral group 
can always be approximated by a snugly embedded discrete \pdash toral group. A more compact terminology will be helpful. 

\begin{definition}   \label{defn: p-discretization}
Let $\Pbold$ be a \pdash toral subgroup of~$G$, and let $P\subgroupeq \Pbold$ be a snugly embedded discrete \pdash toral subgroup with $\Pbar=\Pbold$. We say that $P$ is a \defining{\pdash discretization} 
of~$\Pbold$. 
\end{definition}

In particular, a \pdash discretization of $\Pbold$ is characterized by being a maximal discrete \pdash toral subgroup of~$\Pbold$. 

\begin{example}\label{example: parametrization of dptas}
A torus has only one \pdash discretization, namely the subgroup consisting of all \pdash torsion elements. Similarly,  Proposition~\ref{proposition: existence of maximal}(2) establishes that any abelian \pdash toral group has a unique \pdash discretization.

In general, however, a \pdash toral group $\Pbold$ that has 
multiple components has many \pdash discretizations. If $P\subgroupeq \Pbold$ is one such, then the others are all conjugate to~$P$ 
in~$\Pbold$ by Proposition~\ref{proposition: existence of maximal}(2). The stabilizer of $P$ is $N_\Pbold(P)$, so the approximations are parametrized by $\Pbold/N_\Pbold(P)$. (See also Remark~\ref{remark: infinite p-discretizations}.)

The simplest nontrivial example with more than one discretization is the $2$-toral group 
$\Pbold=O(2)\cong S^{1}\rtimes\{\pm 1\}$, where $-1$ is represented by
reflection over the $y$-axis. 
An obvious $2$-discretization is given by the subgroup $P=\Ztwoinfinity\rtimes\{\pm 1\}$. 
A direct matrix calculation shows that $N_\Pbold(P)=P$, so in fact the $2$-discretizations are parametrized by $\Pbold/P\cong S^1/\left(\Ztwoinfinity\right)$. 
The other parametrizations are given by 
\begin{equation}    \label{eq: other discretization}
P'=\left(\strut \Ztwoinfinity\times\{1\}\right)\ 
\sqcup 
\ \left(\strut \xi \cdot \Ztwoinfinity\times\{-1\}\right)
\end{equation}
where $\xi$ is any fixed element of~$S^1$. 
And indeed, the proof of \cite[9.3]{BLO-Discrete} establishes that, in general, the different \pdash discretizations of a \pdash toral group $\Pbold$ can be obtained by conjugation by an element of the torus of~$\Pbold$.
\end{example}

We close this section by observing that, as in \eqref{eq: other discretization}, any \pdash discretization must start with the unique \pdash discretization of the torus. 

\begin{lemma}   \label{lemma: discretizations contain T_p}
If $\Pbold$ is a \pdash toral group with maximal torus~$\Tbold$, and $T_p$ denotes the \pdash torsion elements of $\Tbold$, then any \pdash discretization of $\Pbold$ must contain~$T_p$. 
\end{lemma}

\begin{proof}
By Proposition~\ref{proposition: existence of maximal}, $T_p$ can be expanded to a \pdash discretization $P'$ of~$\Pbold$, and $P'$ is conjugate to $P$ in~$\Pbold$. However, $T_p$ is a normal subgroup of~$\Pbold$ and hence is stabilized by the conjugation. Thus $T_p\subseteq P$ as well. 
\end{proof}


\medskip
\section{\pdash discretizations of pairs}
\label{section: relative dptas}

Let $\Pbold$ be a \pdash toral group. 
Proposition~\ref{proposition: existence of maximal} tells us that all \pdash discretizations of~$\Pbold$ are conjugate in~$\Pbold$. It also tells us that if $\Qbold\subgroupeq\Pbold$ is a \pdash toral subgroup 
and $Q$ is a \pdash discretization of~$\Qbold$, then $Q$ can be expanded to a \pdash discretization $P$ of~$\Pbold$. 
However, it is likely that there is more than one way to expand~$Q$;  that is, there can be different pairs of discrete \pdash toral subgroups $(Q,P_1)$ and $(Q,P_2)$ that are \pdash discretizations for the pair $(\Qbold,\Pbold)$. The main goal of this section is to establish the following proposition, which can be thought of as a uniqueness statement about \pdash discretizations
of pairs. The point is that $P_1$ and $P_2$ are conjugate in $\Pbold$ by an element that fixes~$Q$.
Hence the pair $(Q,P_1)$ is conjugate in $\Pbold$ to the pair $(Q,P_2)$. 

\begin{proposition} \label{proposition: uniqueness of relative dptas} 
If $P_1$ and $P_2$ are \pdash discretizations of~$\Pbold$, and $Q\subgroupeq P_1\cap P_2$, then there exists $\ybold\in C_{\Pbold}(Q)$ such that $c_{\ybold}\left(P_1\right)=P_2$. 
\end{proposition}

Our approach is based on a non-canonical (and non-topological) splitting of \pdash toral groups, for which we use the following
standard homological lemma. 

\begin{lemma}    \label{lemma: splitting of ZG-mod} 
Let $\group$ be a finite group, and let 
\[
0\longrightarrow I \longrightarrow X \longrightarrow V \longrightarrow 0
\]
be a short exact sequence of $\integers[\group]$-modules. If $V$ is uniquely $|\group|$-divisible and $I$ is an injective $\integers$-module, then there is a splitting 
$X\cong I\times V$ as $\integers[\group]$-modules. 
\end{lemma}

\begin{proof}
Because $I$ is an injective abelian group, there is a retraction of abelian groups $r\colon X\rightarrow I$,
which in turn defines a section $s\colon V\rightarrow X$ of abelian groups. However, $\Hom_{\integers}(V,X)$ is uniquely $|\group|$-divisible
(because $V$ is),
so we can define a new section $\swiggle\colon V\rightarrow X$
by averaging over the elements of~$\group$,
\[
\swiggle\,\definedas \frac{1}{|\group|}\sum_{y\in \group} y^{-1}sy. 
\]
Then $\swiggle$ is a section of $X\rightarrow V$ as $\integers[\group]$-modules, which establishes the lemma. 
\end{proof}

We use the following notation for the splitting result below. 
Let $\Tbold=\reals^r/\integers^r$ be a rank~$r$ torus, 
whose subgroup of torsion elements is denoted 
$T_{\rationals}\definedas \rationals^r/\integers^r$.
The quotient of $\Tbold$ by the torsion elements is denoted  
$T_{\infty}\definedas \Tbold/T_{\rationals}$. If $p$ is a prime, then the \pdash torsion subgroup of $\Tbold$ is denoted $T_p\definedas\left(\Zpinfinity\right)^r$,
and we write $T_{p'}$ for the subgroup of $\Tbold$ consisting of torsion elements of order prime to~$p$, i.e. the product of all the subgroups $\left(\Zqinfinity\right)^r$ over primes $q\neq p$. 

\begin{lemma}   \label{lemma: retraction to dpta} 
Let $\Pbold$ be a \pdash toral group with \pdash discretization~$P\subgroupeq \Pbold$. There exists a (non-canonical, discontinuous) group homomorphism $\Pbold\rightarrow P$ that splits the inclusion $P\hookrightarrow\Pbold$. Any such splitting has the property that if $P'$ is another \pdash discretization of $\Pbold$, then $P'\rightarrow\Pbold\rightarrow P$ is an isomorphism.
\end{lemma}

\begin{proof}
Let $\Tbold$ be the maximal torus of~$\Pbold$ and let $\group=\pi_{0}\Pbold=\Pbold/\Tbold$, a finite \pdash group. By Lemma~\ref{lemma: splitting of ZG-mod} applied to $\Tbold$ (considered as a discrete group), there is a split short exact sequence of $\integers[\group]$-modules 
\[
0\longrightarrow T_{p}
 \longrightarrow \Tbold
 \longrightarrow T_{p'}\times T_{\infty}
 \longrightarrow 0.
\]
Note that by Lemma~\ref{lemma: discretizations contain T_p}, $T_p\subgroupeq P$. Further, because $T_{p'}\times T_{\infty}$ is split from $\Tbold$ as a $\integers[\group]$-module, we know $T_{p'}\times T_{\infty}$ is normal in~$\Pbold$ and we can define the quotient 
$\Pwiggle\definedas \Pbold/\left(T_{p'}\times T_{\infty}\right)$, which is 
a discrete \pdash toral group. 
(We note here that we have completely discarded the topology on the torus. 
The key tool resulting from Lemma~\ref{lemma: retraction to dpta}
is Lemma~\ref{lemma: outer transporters}, and the topology is not needed there.) 

Consider the commutative ladder of exact sequences
\begin{equation} \label{eq: 3 by 3 diagram}
\begin{gathered}
\xymatrix{
0\ar[r] & T_p\ar[r]\ar[d] 
        & P\ar[r]\ar[d]_-{i}
        & \group\ar[r]\ar[d]^{=}
        & 0\\
0 \ar[r]& \underbrace{T_{p}\times T_{p'}\times T_{\infty}}_{\Tbold}
          \ar[r] \ar[d]
        & \Pbold\ar[r]\ar[d]_-{q}
        & \group\ar[r]\ar[d]^{=}
        & 0\\
0\ar[r] & T_p\ar[r]
        & \Pwiggle\ar[r]
        & \group\ar[r]
        & 0.
}
\end{gathered}
\end{equation}
By construction, the compositions of the two maps in the first and third columns
are identity maps on $T_p$ and $\group$, respectively.  
Hence the composite $f\definedas q\circ i\colon P\rightarrow \Pwiggle$ is an isomorphism.
Then $f^{-1}\circ q\colon \Pbold \rightarrow P$ is the required group homomorphism splitting the inclusion $P\subgroupeq \Pbold$.

If $P'\subgroup\Pbold$ is another \pdash discretization of~$\Pbold$, then $P'$ also contains $T_p$ (Lemma~\ref{lemma: discretizations contain T_p}). Therefore we can substitute $P'$ for $P$ in \eqref{eq: 3 by 3 diagram} and the composite $P'\rightarrow\Pbold\rightarrow\Pwiggle$ will still be an isomorphism. Hence composing with the isomorphism 
$f^{-1}\colon \Pwiggle\rightarrow P$ finishes the proof. 
\end{proof}

Using the splitting, we are able to show a sense in which a \pdash discretization $P\subseteq\Pbold$ is able to capture conjugation information present in~$\Pbold$.  The statement below is a slight generalization of \cite[Lemma~9.4(a)]{BLO-Discrete} and a couple of statements in its proof.

\begin{lemma}  \label{lemma: outer transporters}
Let $\Pbold$ be a \pdash toral subgroup with \pdash discretization~$P$. Let $Q_1$ and $Q_2$ be subgroups of~$P$, and suppose that a group homomorphism $f\colon Q_1\rightarrow Q_2$ is induced by conjugation in~$\Pbold$. Then $f$ can be induced by conjugation in~$P$. 
\end{lemma}

\begin{proof}
Suppose that $f\colon Q_1\rightarrow Q_2$ is given by conjugation by $\ybold\in\Pbold$. Let $r\colon\Pbold\rightarrow P$ be the retraction provided by Lemma~\ref{lemma: retraction to dpta}, and consider the commutative diagram
\begin{equation*} 
\begin{gathered}
\xymatrix{
Q_1\ \ar@{^(->}[r]\ar[d]_{f=c_\ybold}  
        & P\ \ar@{^(->}[r]
        & \Pbold\ar[r]_{r}\ar[d]_{c_\ybold}
        & P\ar[d]^{c_{r(\ybold)}}\\
Q_2\  \ar@{^(->}[r] 
        & P\ \ar@{^(->}[r] 
        & \Pbold\ar[r]_{r}
        & P . 
}
\end{gathered}
\end{equation*}
Although we cannot fill in the rectangle, because conjugation by $\ybold$ may not take $P$ to~$P$, we do know that (by assumption) that $Q_1$ and $Q_2$ are contained in~$P$. Since $P\hookrightarrow \Pbold\xrightarrow{r} P$ is the identity map of~$P$, the compositions across the top and bottom rows corestrict to the identity maps on $Q_1$ and~$Q_2$, respectively. Therefore conjugation by $\ybold\in\Pbold$ and $r(\ybold)\in P$
induce the same map $f\colon Q_1\rightarrow Q_2$. 
\end{proof}

We now have all the tools we need to establish 
Proposition~\ref{proposition: uniqueness of relative dptas}. 

\begin{proof}[Proof of Proposition~\ref{proposition: uniqueness of relative dptas}]
Since $P_1$ and $P_2$ are both \pdash discretizations of~$\Pbold$, they are conjugate in~$\Pbold$. Choose $\ybold\in\Pbold$ such that $c_{\ybold}(P_1)=P_2$. It is possible that $c_{\ybold}$ does not stabilize~$Q$, so let 
$Q'=c_{\ybold}(Q)$. Then $Q$ and $Q'$ are both subgroups of~$P_2$, and $c_{\ybold}\colon Q\rightarrow Q'$. By 
Lemma~\ref{lemma: outer transporters}, there exists 
$x\in P_2$ such that $c_{x}=c_{\ybold}\colon Q\rightarrow Q'$. 
Define $\ybold'=\ybold\cdot x^{-1}$. Then $\ybold'$ still conjugates $P_1$ to $P_2$, but $\ybold'$ centralizes~$Q$. 
\end{proof}

In the remainder of this section, we give two applications
of Proposition~\ref{proposition: uniqueness of relative dptas}.
First, we prove that the outer automorphism group in $G$ of a \pdash discretization is the same as that of its closure. 
This is proved in \cite[Lemma~9.4]{BLO-Discrete} using a different point of view.
Second, we prove that for the purpose of understanding the mod~$p$ homology of classifying spaces, centralizers and 
normalizers can be computed either in a discrete \pdash toral group or a continuous \pdash toral group. These two results are the base cases for inductions to establish the corresponding results for chains of subgroups in Section~\ref{section: normalizers}. 

\begin{lemma} \label{lemma: outs for one group}
Let $P$ be a \pdash discretization of a \pdash toral subgroup $\Pbold$ of~$G$. Then $\Out_G(P)\cong\Out_G(\Pbold)$.
\end{lemma}

\begin{proof}
We want to prove that the natural map
\[
\Out_G(P)\definedas\frac{N_G(P)}{C_G(P)\cdot P}
\,\xrightarrow{\,\quad\, }\,
\frac{N_G(\Pbold)}{C_G(\Pbold)\cdot \Pbold}=:\Out_G(\Pbold)
\]
is an isomorphism. To show that it is an epimorphism, suppose that $\nbold\in N_G(\Pbold)$. Because $c_\nbold(P)$ and $P$ are both discretizations of~$\Pbold$, there exists $\xbold\in\Pbold$ such that $c_\xbold(c_\nbold P)=P$. Therefore $\nbold\cdot \xbold\in N_G(P)$,
and it represents the same class as $\nbold$ in $\Out_G(\Pbold)$. 

To show injectivity, first suppose that 
$n\in N_G(P) \cap \Pbold$.
We would like to show that $n$ is already in 
$C_G(P)\cdot P$.  However, Lemma~\ref{lemma: outer transporters} tells us that the automorphism of $P$ induced by $n$ can be induced by some $y\in P$. 
Hence $n\cdot y^{-1}=c\in C_G(P)$ and 
$n=c\cdot y\in C_G(P)\cdot P$, as required. 

To finish, suppose that 
$n\in N_G(P)\cap \left[C_G(\Pbold)\cdot \Pbold\right]$, say 
$n=\cbold\cdot\xbold$ with $\cbold\in C_G(\Pbold)=C_G(P)$
and $\xbold\in\Pbold$. Then $\xbold=n\cdot \cbold^{-1}$ normalizes~$P$. The previous argument shows that $\xbold\in C_G(P)\cdot P$, and hence $n=\cbold\cdot\xbold\in C_G(P)\cdot P$, as required.  
\end{proof}

\begin{remark}  \label{remark: infinite p-discretizations}
Observe that $N_\Pbold(P)=N_G(P) \cap \Pbold$, and the proof of 
Lemma~\ref{lemma: outs for one group} establishes that 
$N_G(P) \cap \Pbold=\left(\strut P\cdot C_G(P)\right)\cap\Pbold$. Since $C_G(P)\cap\Pbold=C_G(\Pbold)\cap\Pbold=Z(\Pbold)$, we have actually 
proved that if $P$ is a \pdash discretization of~$\Pbold$, then $N_\Pbold(P)=P\cdot Z(\Pbold)$. This gives a refinement to the discussion of Example~\ref{example: parametrization of dptas}: 
\pdash discretizations of $\Pbold$ are parametrized by $\Pbold/N_\Pbold(P)=\Pbold/(P\cdot Z(\Pbold))$. We recover the result that if $\Pbold$ is abelian ($Z(\Pbold)=\Pbold$), then the \pdash discretization is unique. 
Indeed, if the torus is central in $\Pbold$ then there is a unique \pdash discretization of~$\Pbold$, and otherwise there are infinitely many.
\end{remark}

For our final result of this section, note that if $\Qbold\subgroupeq\Pbold$ is an inclusion of \pdash toral subgroups, then both $N_{\Pbold}(\Qbold)$ and
$C_{\Pbold}(\Qbold)$ (and hence $Z(\Qbold)$) are \pdash toral (\cite[Lemma~A.3]{JMO}).  

\begin{proposition}   \label{proposition: discrete to continuous}
Let $Q\subgroupeq P$ be \pdash discretizations of \pdash toral groups $\Qbold\subseteq\Pbold$. Then 
$C_{P}(Q)$ is a \pdash discretization of $C_{\Pbold}(\Qbold)$, and
$N_{P}(Q)$ is a \pdash discretization of $N_{\Pbold}(\Qbold)$. 
\end{proposition}

\begin{proof}
Let $D$ be a \pdash discretization of~$C_{\Pbold}(\Qbold)$; we note that $Z(Q)$ is necessarily contained in~$D$, since $Z(\Qbold)$ has only one \pdash discretization. 
Since $D$ commutes with~$Q$, their product $D\cdot Q$ is a discrete \pdash toral subgroup of 
$N_{\Pbold}(\Qbold)$. We can expand 
$D\cdot Q$ to a \pdash discretization $N$ of 
$N_{\Pbold}(\Qbold)$, 
and then enlarge $N$ to a \pdash discretization $P'$ of~$\Pbold$. So we have 
compatible \pdash discretizations 
\begin{equation*} 
\begin{gathered}
\xymatrix{
        \Largestrut D \ar[r]\ar@{^(->}[d] 
        & D\cdot Q \ar[r]
        & \Largestrut N\ar[r]\ar@{^(->}[d]
        & \Largestrut P'\ar@{^(->}[d]\\
        C_{\Pbold}(\Qbold) \ar[rr]
        && N_{\Pbold}(\Qbold)\ar[r]
        & \Pbold.
}
\end{gathered}
\end{equation*}

By construction, $Q\subseteq N\cap \Qbold$, and since $Q$ is a maximal discrete \pdash toral subgroup of~$\Qbold$, we know that $Q=N\cap\Qbold$. Therefore $N$ normalizes~$Q$ (because $N$ normalizes both itself and $\Qbold$), so $N\subgroupeq N_{P'}(Q)\subgroupeq N_\Pbold(\Qbold)$. But $N$ is a maximal discrete \pdash toral subgroup of $N_\Pbold(\Qbold)$, so in fact $N=N_{P'}(Q)$. 
Similarly, we have
$D\subseteq C_{P'}(Q)\subseteq C_{\Pbold}(\Qbold)$ and maximality gives us
$D=C_{P'}(Q)$.

However, we have another \pdash discretization of~$\Pbold$, namely~$P$. Notice that $Q$ is contained both $P$ (by assumption) and~$P'$ (by construction), 
so Proposition~\ref{proposition: uniqueness of relative dptas} 
tells us that there exists $\ybold\in C_{\Pbold}(Q)$ with $c_{\ybold}(P')=P$. 
We obtain the two commutative diagrams
\begin{equation*} 
\begin{gathered}
\xymatrix{
\Largestrut C_{P'}(Q)\ar[r]_{c_{\ybold}}\ar@{^(->}[d]
        & \Largestrut C_{P}(Q)\ar@{^(->}[d]\\
C_{\Pbold}(\Qbold)\ar[r]_{c_{\ybold}}
        & C_{\Pbold}(\Qbold)
}
\hspace{5em}
\xymatrix{
\Largestrut N_{P'}(Q)\ar[r]_{c_{\ybold}}\ar@{^(->}[d]
        & \Largestrut N_{P}(Q)\ar@{^(->}[d]\\
N_{\Pbold}(\Qbold)\ar[r]_{c_{\ybold}}
        & N_{\Pbold}(\Qbold).
}
\end{gathered}
\end{equation*}
The left vertical arrows are \pdash discretizations by construction, therefore the right vertical arrows are \pdash discretizations as well. 
\end{proof}

\begin{corollary}   \label{cor: center is snug}
If $P$ is a \pdash discretization of~$\Pbold$, 
then $Z(P)$ is a \pdash discretization of~$Z(\Pbold)$. 
\end{corollary}

\begin{proof}
Apply Proposition~\ref{proposition: discrete to continuous} with 
$P=Q$. 
\end{proof}


\medskip

\section{Chains of \pdash centric, \pdash stubborn subgroups of~$G$}
\label{section: what are the chains}

In \cite{BLO-Discrete}, 
Broto, Levi and Oliver construct a saturated fusion system associated to a compact Lie group~$G$, denoted $\Fcal_S(G)$, where $S$ is a maximal discrete \pdash toral subgroup of~$G$. It is a category whose objects are the subgroups of~$S$ and whose morphisms are homomorphisms induced by conjugation in~$G$. 
The purpose of this section is to
compare the collection of $\Fcal_S(G)$-centric, $\Fcal_S(G)$-radical subgroups (Definition~\ref{defn: centric, radical, stubborn}) with the analogous collection of continuous \pdash toral subgroups, namely the \pdash centric, \pdash stubborn subgroups. In our forthcoming normalizer decomposition for \pdash local compact groups \cite{BCGLS-Main}, the indexing category will be conjugacy classes of chains of $\Fcal_S(G)$-centric, $\Fcal_S(G)$-radical subgroups of~$S$. In this section, we show that when $\pi_0G$ is a \pdash group, the set of such conjugacy classes is in one-to-one correspondence with conjugacy classes of chains of \pdash centric, \pdash stubborn subgroups of~$G$ (Theorem~\ref{theorem: identification of chains}). Further, even when $\pi_0G$ is not a \pdash group, there is still an injection from the first set to the second. 

First we need some definitions, taken from the definitions for a fusion system (see \cite[Def.~2.6 and pp~380]{BLO-Discrete} and \cite[Def.~1.3]{JMO}). For streamlined notation, we suppress both $G$ and the maximal discrete \pdash toral subgroup~$S$.

\begin{definition}     \label{defn: centric, radical, stubborn}
Fix a compact Lie group~$G$ and a maximal discrete \pdash toral 
subgroup $S\subgroupeq G$. 
\begin{enumerate}
    \item For discrete \pdash toral groups
         \label{item: discrete}
\begin{enumerate}
\item A subgroup $P\subgroupeq S$ is \defining{\fusion-centric} if whenever $Q\subgroupeq S$ is $G$-conjugate to~$P$, we have $C_S(Q)=Z(P)$. (In particular, $C_S(P)=Z(P)$.) 
\item A subgroup $P\subgroupeq S$ is \defining{\fusion-radical} if $\Out_{G}(P)\definedas N_G(P)/\left[C_G(P)\cdot P\right]$ 
    has no nontrivial normal subgroups. 
\end{enumerate}
\item For continuous \pdash toral groups
     \label{item: continuous}
\begin{enumerate} 
\item
\label{item: continuous pcentric in G}
    A \pdash toral subgroup $\Pbold\subgroupeq G$ is
    \defining{\pdash centric in~$G$} if $Z(\Pbold)$ is a maximal \pdash toral subgroup of $C_G(\Pbold)$.
\item 
\label{item: continuous pstubborn in G}
    A \pdash toral subgroup $\Pbold\subgroupeq G$ is
    \defining{\pdash stubborn in~$G$} if $N_G(\Pbold)/\Pbold$ is finite
    and has no nontrivial normal \pdash subgroups.
\end{enumerate}
\end{enumerate}
\end{definition}


Note that although the concepts of $\Fcal$-centric and $\Fcal$-radical depend on both $S$ and~$G$, we omit them from the notation because $S$ and $G$ are always clear from context, and the omission gives a slimmer notation. We also observe that the properties of being \pdash centric and \pdash stubborn are closed under $G$-conjugation since $c_g(C_G(\Pbold))=C_G(c_g(\Pbold))$ and $c_g(N_G(\Pbold))=N_G(c_g(\Pbold))$ for any $g\in G$.

\begin{remark}\label{rmk:radical-vs-stubborn}
A \pdash toral subgroup $\Pbold\subgroupeq G$ is \defining{\pdash radical} if $N_G(\Pbold)/\Pbold$ has no nontrivial normal \pdash toral subgroups (no assumption that $N_G(\Pbold)/\Pbold$ is finite). For finite groups, the collection of \pdash radical subgroups was introduced by Bouc. In the compact Lie case, 
the \pdash radical subgroups were featured in \cite{Libman-Minami}.

However, when $\Pbold\subgroupeq G$ is \pdash centric, then \pdash radical and \pdash stubborn become equivalent, because
$\Out_G(\Pbold)$ is finite  (\cite[Lemma~9.4]{BLO-Discrete}) and
the short exact sequence 
\[
0\rightarrow C_G(\Pbold)/Z(\Pbold)\rightarrow N_G(\Pbold)/\Pbold \rightarrow \Out_G(\Pbold)\rightarrow 0
\]
has $C_G(\Pbold)/Z(\Pbold)$ finite of order prime to~$p$. 
\end{remark}

With the necessary vocabulary in hand, we are able to state the main theorem for this section. 

\begin{theorem}
\label{theorem: identification of chains}
\chainstheoremtext
\end{theorem}

The first task is to show that the map of 
Theorem~\ref{theorem: identification of chains} actually exists. 
That is, we need to establish that the closure of a discrete \pdash toral subgroup of $S$ that is \fusion-centric and \fusion-radical is \pdash centric and \pdash stubborn. (It is certainly \pdash toral.) 
To use the results of Section~\ref{section: relative dptas}, we need to know that the discrete \pdash toral groups we are dealing with are snugly embedded. The following lemma can be found in {\cite[Corollary 3.5, Lemma 9.9]{BLO-Discrete}}, but we give an elementary proof here that does not use the bullet construction. 

\begin{lemma}
\label{lemma: bullet-snugly-embedded}
Let $G$ be a compact Lie group with maximal discrete
\pdash toral subgroup~$S$, and let $P$ be a subgroup of~$S$.
If $P$ is \fusion-centric and \fusion-radical, then $P$ is snugly embedded. 
\end{lemma}

\begin{proof}
Let $\Pbold=\Pbar$, and expand $P$ to a \pdash discretization $Q'$ of~$\Pbold$. To prove that $P$ is snugly embedded, we would like to prove that $Q'=P$. Expand $Q'$ further to a \pdash discretization $S'$ of~$\Sbold$, so that we have 
$P\subseteq Q'\subgroupeq S'\subgroupeq \Sbold$. Using Proposition~\ref{proposition: uniqueness of relative dptas}
with $P\subgroupeq S\cap S'$, choose $\ybold\in C_\Sbold(P)$ such that $c_\ybold(S')=S$, and let $Q=c_\ybold(Q')$. Now we have $P\subseteq Q\subgroupeq S\subgroupeq \Sbold$ and $Q$ is a \pdash discretization of~$\Pbold$, and our goal has become to prove $Q=P$.

Consider the homomorphism 
\begin{equation}   \label{eq: Q to G}
N_Q(P)/Z(P)\rightarrow N_G(P)/C_G(P).
\end{equation}
Because $P$ is \fusion-centric by assumption, $C_S(P)=Z(P)$. Therefore the centralizer of~$P$ in $Q\subgroupeq S$ is $Z(P)$ as well, and \eqref{eq: Q to G} is a monomorphism. We assert that the image is a normal subgroup
of~$N_G(P)/C_G(P)$. To prove this, we must take an element $g\in N_G(P)$ and prove that we can adjust $g$ by an element of $x\in C_G(P)$ so that $g\cdot x$ normalizes~$N_Q(P)$. Given that $g\cdot x$ would certainly normalize~$P$, it is sufficient to construct $x$ so that $g\cdot x$ normalizes~$Q$. 

Let $Q''=c_g(Q)$. Because $g\in N_G(P)\subgroupeq N_G(\Pbold)$, the groups $Q''$ and $Q$ are both \pdash discretizations of~$\Pbold$, and we have $P\subgroupeq Q\cap Q''$. Proposition~\ref{proposition: uniqueness of relative dptas} gives us an element $x\in C_\Pbold(P)$ such that 
$c_x(Q'')=Q$. Therefore $g\cdot x$
normalizes both $Q$ and~$P$ and we conclude that 
\eqref{eq: Q to G} is the inclusion of a normal subgroup. Taking the quotient on both sides by~$P$, we find that $N_Q(P)/P$ is a normal subgroup of $\Out_G(P)$. 

However, we have assumed that $P$ is \fusion-radical, meaning that $\Out_G(P)$ has no nontrivial \pdash subgroups. Therefore $N_Q(P)/P$ must be the trivial group, that is, $N_Q(P)=P$. Because $P$ and $Q$ are discrete \pdash toral groups, $N_Q(P)=P$ implies that the inclusion of $P$ into $Q$ cannot be proper \cite[Lemma~1.8]{BLO-Discrete}, so $P=Q$. Hence $P$ is a maximal discrete \pdash toral subgroup of $\Pbold$, as required. 
\end{proof}

Now that we know that \fusion-centric and \fusion-radical subgroups of $G$ must be snug, we can 
use the results of Section~\ref{section: relative dptas}.
Our next proposition shows that the map of Theorem~\ref{theorem: identification of chains} can be defined. 
That is, we show that the closure of an 
\fusion-centric and \fusion-radical discrete \pdash toral subgroup is in fact \pdash centric and \pdash stubborn. (See also the argument given in \cite[Prop~8.4 and Lemma~9.6]{BLO-Discrete} for \pdash centricity.)

\begin{proposition}   \label{proposition: p-stubborn subgroups}
If $P$ is a \pdash discretization of $\Pbold$, and $P$ is \fusion-centric and \fusion-radical, then $\Pbold$ is \pdash centric and \pdash stubborn.
\end{proposition}

\begin{proof}
To show that $\Pbold$ is \pdash centric, we must show that $Z(\Pbold)$ is a maximal \pdash toral subgroup of~$C_G(\Pbold)$. 
To see this, suppose that $\Hbold\subgroupeq C_G(\Pbold)$ 
is a maximal \pdash toral subgroup. 
Then $\Hbold$ necessarily contains $Z(\Pbold)$, because all choices for $\Hbold$ are conjugate in $C_G(\Pbold)$ and 
$Z(\Pbold)\triangleleft \, C_G(\Pbold)$.
We construct a \pdash discretization of~$\Hbold$ by expanding $Z(P)\subgroupeq \Hbold$ to
a \pdash discretization $H$ of $\Hbold$. Then we further expand the discrete \pdash toral subgroup $H\cdot P$ to a maximal discrete \pdash toral subgroup $S'$ of~$G$. (Note that $S'$ does not have to have the same closure as~$S$.) 

All maximal discrete \pdash toral subgroups of $G$ are conjugate, so there exists $g\in G$ with $c_g(S')=S$. Let $Q=c_g(P)\subgroupeq S$ and $J=c_g(H)\subgroupeq S$
so we have the following picture:
\begin{equation*} 
\begin{gathered}
\xymatrix{
Z(P)               \ar[r]\ar[d]_{c_g}
        & H        \ar[r]\ar[d]_{c_g} 
        & P\cdot H \ar[r]\ar[d]_{c_g}
        & S'       \ar[r]\ar[d]_{c_g}
        & G        \ar[d]_{c_g}\\
Z(Q)               \ar[r]
        & J        \ar[r] 
        & Q\cdot J\ar[r]
        & S         \ar[r]
        & G.
}
\end{gathered}
\end{equation*}
By construction, $\Hbold=C_{\Sbold'}(\Pbold)=C_{\Sbold'}(P)$, and further, $H$ is a \pdash discretization of~$\Hbold$ and $H\subgroupeq S'$. This means $H=C_{S'}(P)$, and so $J=C_S(Q)$. 

But we know that $Q\subgroupeq S$ is 
$G$-conjugate to~$P$, and by definition of \fusion-centric,
that implies $Z(Q)=C_{S}(Q)=J$. Hence $Z(P)=H$ as well. 
Lastly, because $P$ is \fusion-centric and \fusion-radical,
$P$ is snug by
Lemma~\ref{lemma: bullet-snugly-embedded}, and therefore   
$Z(P)$ is a \pdash discretization of $Z(\Pbold)$
by Corollary~\ref{cor: center is snug}. Since $H$ is a \pdash discretization of $\Hbold$ by construction, and $H=Z(P)$, we find that $\Hbold=Z(\Pbold)$, as required to prove that $\Pbold$ is \pdash centric. 

To show that $\Pbold$ is \pdash stubborn, we must show that $N_{G}\Pbold/\Pbold$ is finite and contains no nontrivial normal \pdash subgroups. Consider the short exact sequence 
\begin{equation} \label{eq: normalizer versus Out}
1\longrightarrow C_G(\Pbold)/Z(\Pbold) 
 \longrightarrow N_G(\Pbold)/\Pbold
 \longrightarrow 
     \underbrace{N_G(\Pbold)/\left[\strut C_G(\Pbold)\cdot \Pbold\right]}_{\Out_G(\Pbold)}
 \longrightarrow 1.
\end{equation}
The right-hand term is $\Out_G(\Pbold)$, which by 
Lemma~\ref{lemma: outs for one group}
is isomorphic to $\Out_{G}(P)$. 
The definition of 
\fusion-radical tells us that 
$\Out_G(P)$ contains no nontrivial normal \pdash subgroups, and hence the same is true of
$\Out_G(\Pbold)$.

Turning our attention to the left-hand term in~\eqref{eq: normalizer versus Out}, 
we have proved that $Z(\Pbold)$ is a maximal \pdash toral subgroup of $C_G(\Pbold)$. The quotient $C_G(\Pbold)/Z(\Pbold)$ is a compact Lie group, but cannot contain an $S^1$ by maximality of~$Z(\Pbold)$. Hence $C_G(\Pbold)/Z(\Pbold)$ is finite. Again by maximality of~$Z(\Pbold)$, we know $C_G(\Pbold)/Z(\Pbold)$ has no \pdash torsion elements, so it has order prime to~$p$. As a consequence, the image of a nontrivial normal \pdash subgroup of $N_G(\Pbold)/\Pbold$ would be a nontrivial \pdash subgroup of $\Out_G(\Pbold)=\Out_{G}(P)$, a contradiction of the assumption that $P$ is \fusion-radical.  

It remains to establish that $N_{G}(\Pbold)/\Pbold$ is finite, for which is it sufficient to know that the right-hand term, $\Out_G(\Pbold)\cong\Out_G(P)$, is finite. If $\Out_G(P)$ is not finite, then it has a nontrivial torus, and therefore an infinite torsion subgroup. However, any torsion subgroup of $\Out_G(P)$ is finite
\cite[Prop.~1.5]{BLO-Discrete}.
Therefore $\Out_G(\Pbold)$ is finite, and hence $N_G(\Pbold)/\Pbold$ is finite with no nontrivial normal \pdash subgroups, meaning that $\Pbold$ is \pdash stubborn, as required. 
\end{proof}

So far, our progress toward proving Theorem~\ref{theorem: identification of chains} is to establish that the statement makes sense: the function actually exists! Next we establish that the function is injective. 

\begin{lemma} \label{lemma: monomorphism}
Let $\Sbold\subgroupeq G$ be a maximal \pdash toral subgroup of the Lie group~$G$, and fix a \pdash discretization $S$ of~$\Sbold$. Suppose that 
$P_0\subgroupeq \ldots\subgroupeq P_{k}$
and $Q_0\subgroupeq \ldots\subgroupeq Q_{k}$
are two chains of snug discrete \pdash toral subgroups of~$S$ that both have 
$\Pbold_0\subgroupeq \ldots\subgroupeq \Pbold_{k}$
as their closure. Then 
$P_0\subgroupeq \ldots\subgroupeq P_{k}$
and $Q_0\subgroupeq \ldots\subgroupeq Q_{k}$ are 
conjugate in~$\Pbold_k$ (and hence necessarily in~$G$). 
\end{lemma}

\begin{proof}
We induct on~$k$. The base case $k=0$ is true because $P_0$ and $Q_0$ are both \pdash discretizations of~$\Pbold_0$, and hence are conjugate in~$\Pbold_0$. Now suppose that  
$P_0\subgroupeq \ldots\subgroupeq P_{k-1}$
and 
$Q_0\subgroupeq \ldots\subgroupeq Q_{k-1}$ 
are conjugate by $\xbold\in\Pbold_{k-1}$.  
Then $c_\xbold(P_k)$ and $Q_k$ are \pdash discretizations of~$\Pbold_k$, and $Q_{k-1}\subgroupeq c_\xbold(P_k)\cap Q_k$.
By Proposition~\ref{proposition: uniqueness of relative dptas}, 
there exists $\ybold\in C_{\Pbold_k}(Q_{k-1})$ such that 
$c_\ybold\left(c_\xbold(P_k)\right)=Q_k$. Hence $\xbold\cdot\ybold\in\Pbold_k$ and conjugates  
$P_0\subgroupeq \ldots\subgroupeq P_{k}$
to
$Q_0\subgroupeq \ldots\subgroupeq Q_{k}$.
\end{proof}

To address the extent to which the function in Theorem~\ref{theorem: identification of chains} is surjective, we first need to check whether it is always possible to obtain a discretization of a given chain in~$\Sbold$ within the chosen \pdash discretization~$S$. 

\begin{lemma} \label{lemma: compatible chains of dptas}
Let $\Sbold\subgroupeq G$ be a maximal \pdash toral subgroup of the Lie group~$G$, and let 
$\Pbold_0\subgroupeq \ldots\subgroupeq \Pbold_{k}$ be a chain of \pdash toral subgroups of~$\Sbold$. Fix a \pdash discretization $S\subgroupeq\Sbold$. Then there exists an $\Sbold$-conjugate of $\Pbold_0\subgroupeq \ldots\subgroupeq \Pbold_{k}$ that has a chain of \pdash discretizations $P_0\subgroupeq \ldots\subgroupeq P_{k}$ contained in~$S$. 
\end{lemma}

\begin{proof}
Choose a \pdash discretization $P_0$ of~$\Pbold_0$, and expand it one group at a time to a chain of \pdash discretizations
$P_0\subgroupeq \ldots\subgroupeq P_k\subgroupeq S'$ of $\Pbold_0\subgroupeq \ldots\subgroupeq \Pbold_{k}\subgroup\Sbold$. Since $S$ and $S'$ are both \pdash discretizations of $\Sbold$, there exists 
$\sbold\in\Sbold$ such that $c_\sbold(S')=S$. 
Then $c_\sbold\left(P_0\subgroupeq \ldots\subgroupeq P_k\right)$ 
is a \pdash discretization inside $S$ of 
$c_\sbold\left(\Pbold_0\subgroupeq \ldots\subgroupeq \Pbold_{k}\right)$.
\end{proof}

\begin{proof}[Proof of Theorem~\ref{theorem: identification of chains}]
Most of the proof is in the preceding results. We have proved that the function exists (Proposition~\ref{proposition: p-stubborn subgroups}) and that it is injective (Lemma~\ref{lemma: monomorphism}). 
Further, Lemma~\ref{lemma: compatible chains of dptas} lays the 
groundwork for an epimorphism statement, 
since a chain in the target has a simultaneous \pdash discretization in the chosen~$S$. To finish the proof of the 
epimorphism statement, we must show that 
\begin{enumerate}
    \item if $\Pbold$ is \pdash stubborn and \pdash centric, with \pdash discretization~$P\subgroupeq S$, then $P$ is \fusion-centric, and 
    \item if in addition $\pi_0 G$ is a \pdash group, then $P$ is also \fusion-radical. 
\end{enumerate}

Suppose that $\Pbold$ is \pdash stubborn and \pdash centric. Then 
$C_G(\Pbold)/Z(\Pbold)\subgroupeq N_G(\Pbold)/\Pbold$ is finite, and has order prime to $p$ because $\Pbold$ is \pdash centric. 
Mapping from $\Sbold$ to~$G$ gives a monomorphism
$C_\Sbold(\Pbold)/Z(\Pbold)\hookrightarrow 
C_G(\Pbold)/Z(\Pbold)$ 
from a \pdash toral group (\cite[A.3]{JMO}) to a finite group 
of order prime to~$p$,
and therefore the map is null. 
We conclude that 
$Z(\Pbold)=C_\Sbold(\Pbold)$.
But by Proposition~\ref{proposition: discrete to continuous} 
and Corollary~\ref{cor: center is snug}, we know that the groups
$Z(P)\subgroupeq C_S(P)$ are \pdash discretizations of 
$Z(\Pbold)=C_\Sbold(\Pbold)$, respectively. Maximality implies that $Z(P)=C_S(P)$, that is, the group $P$ satisfies the required condition to be
\fusion-centric. 

We must also check that if $Q\subgroupeq S$ is $G$-conjugate to~$P$, then $C_S(Q)=Z(Q)$. The subgroup $Q$ is snug, because $P$ is. Let $\Qbold$ be the closure of~$Q$, and observe that $\Qbold$ is  $G$-conjugate to~$\Pbold$, by the same element that takes $Q$ to~$P$. Further, $\Qbold$ is \pdash stubborn and \pdash centric, because $\Pbold$ is, and those properties are preserved by conjugation in~$G$. Since $Q\subgroupeq\Qbold$ is a \pdash discretization, the argument of the previous paragraph shows that $Q$ is \fusion-centric as well.

We must still show that $P$ is \fusion-radical when we know that $\pi_0 G$ is a \pdash group, that is, we must show that 
$\Out_G(P)$ has no nontrivial normal \pdash subgroups. 
By Lemma~\ref{lemma: outs for one group}, 
$\Out_G(P)\cong\Out_G(\Pbold)$, so we can use the short exact sequence 
of~\eqref{eq: normalizer versus Out}. 
The key ingredient is that if $\pi_0 G$ is a \pdash group, then $C_{G}(\Pbold)$ is \pdash toral \cite[A.5]{JMO}, 
and hence $C_{G}(\Pbold)/\Pbold\subgroupeq N_G(\Pbold)/\Pbold$ is
also \pdash toral. However, $\Pbold$ is \pdash stubborn by assumption, so 
$N_G(\Pbold)/\Pbold$ is finite, and $C_{G}(\Pbold)/\Pbold$ is therefore a finite \pdash group. If $\Out_G(\Pbold)$ had a nontrivial normal \pdash subgroup, then its inverse image in $N_G(\Pbold)/\Pbold$ would be a normal \pdash subgroup, in contradiction of the assumption that $\Pbold$ is \pdash stubborn. 
\end{proof}


\medskip

\section{Normalizers}
\label{section: normalizers}

In Section~\ref{section: relative dptas} we studied relative discretizations and proved that 
centralizers and normalizers of \pdash toral groups inside other \pdash toral groups are compatible with \pdash discretizations
(Proposition~\ref{proposition: discrete to continuous}). In this section, we leverage the results of Section~\ref{section: relative dptas} 
to prove that if $P_0\subgroupeq \ldots \subgroupeq P_k$ is a chain of \pdash discretizations of $\Pbold_0\subgroupeq \ldots \subgroupeq \Pbold_k$, then the corresponding map of $G$-normalizers induces a mod~$p$ homology isomorphism on classifying spaces.

\begin{theorem}    \label{theorem: comparison of chains}
\maintheoremtext
\end{theorem}

Our forthcoming work on the normalizer decomposition of a \pdash local compact group will use Theorem~\ref{theorem: comparison of chains} to establish that, when applied to a compact Lie group, our decomposition recovers a version of the theorem of Libman \cite{Libman-Minami} 
using \pdash toral subgroups that are both \pdash centric and \pdash stubborn. 

Our strategy to prove Theorem~\ref{theorem: comparison of chains} is to study the ``outer automorphism" group of a chain separately from the ``inner automorphisms" of the chain.

\begin{definition} \label{defn: outer of a chain}
Let $\newgroup$ be a group, and let
$P_0\subgroupeq \ldots \subgroupeq P_k$
be a chain of subgroups of~$\newgroup$. We define 
$\Out_{\newgroup}(P_0\subgroupeq \ldots \subgroupeq P_k)$ 
as the quotient
\[
\frac{N_{\newgroup}\left(P_0\subgroupeq \ldots \subgroupeq P_k\right)}
{C_{\newgroup}(P_{k})\cdot N_{P_{k}}\left(P_0\subgroupeq \ldots \subgroupeq P_{k}\right)}.
\]
\end{definition}

Note that the definition makes no restriction on the subgroups in the chain. In particular, in the next 
proposition, we establish the relationship between the outer automorphism group of a chain of continuous \pdash toral subgroups and that of a \pdash discretization of the chain.

\begin{proposition} \label{proposition: iso of outer autos discrete to cont}
Let $\Pbold_0\subgroupeq \ldots \subgroupeq \Pbold_k$ be a chain of \pdash toral subgroups of~$G$, and let 
$P_0\subgroupeq \ldots \subgroupeq P_k$ be a chain of \pdash discretizations of the \pdash toral chain. Then inclusion of normalizers induces an isomorphism
\begin{equation}  \label{eq: map of outs}
\underbrace{\Hugestrut
\Out_{G}(P_0\subgroupeq \ldots \subgroupeq P_k)}_{\frac{N_{G}\left(P_0\subgroupeq \ldots \subgroupeq P_k\right)}
{C_{G}(P_{k})\cdot N_{P_{k}}\left(P_0\subgroupeq \ldots \subgroupeq P_{k}\right)}}
\xlongrightarrow{\quad}
\underbrace{\Hugestrut
\Out_{G}(\Pbold_0\subgroupeq \ldots \subgroupeq \Pbold_k)}_{\frac{N_{G}\left(\Pbold_0\subgroupeq \ldots \subgroupeq \Pbold_k\right)}
{C_{G}(\Pbold_{k})\cdot N_{\Pbold_{k}}\left(\Pbold_0\subgroupeq \ldots \subgroupeq \Pbold_{k}\right)}}.
\end{equation}
\end{proposition}

\begin{proof}
To show that the map is an epimorphism, we induct on~$k$. 
The case $k=0$ is Lemma~\ref{lemma: outs for one group}.
For the inductive hypothesis, assume that we have
$g\in N_{G}\left(\Pbold_0\subgroupeq \ldots \subgroupeq \Pbold_k\right)$, and that $g$ stabilizes $P_0\subgroupeq\ldots\subgroupeq P_{k-1}$. 
We need to adjust $g$ to stabilize~$P_k$ as well. 

Suppose that
$c_{g}(P_{k})=P'_{k}$.
Then $P_{k-1}\subgroupeq P_k\cap P'_{k}$, and by 
Proposition~\ref{proposition: uniqueness of relative dptas} we can find $\ybold\in C_{\Pbold_{k}}(P_{k-1})$ that conjugates $P'_k$ to~$P_k$.
Now we have an element $g\cdot \ybold$ that stabilizes $P_0\subgroupeq\ldots \subgroupeq P_k$, 
Since $\ybold$ is in 
$N_{\Pbold_{k}}\left(\Pbold_0\subgroupeq \ldots \subgroupeq \Pbold_{k}\right)$, we conclude that we have an epimorphism
\[
N_{G}\left(P_0\subgroupeq \ldots \subgroupeq P_k\right)
\xlongrightarrow{\rm{\ epi}\ }
\frac{N_{G}\left(\Pbold_0\subgroupeq \ldots \subgroupeq \Pbold_k\right)}
{N_{\Pbold_{k}}\left(\Pbold_0\subgroupeq \ldots \subgroupeq \Pbold_{k}\right) },
\]
and therefore 
\eqref{eq: map of outs} is also surjective.

To establish injectivity, consider
$n\in N_{G}\left(P_0\subgroupeq \ldots \subgroupeq P_k\right)$ such that $[n]$ is in the kernel
of~\eqref{eq: map of outs}. We can write $n=c\cdot \xbold$ where 
$c\in C_{G}(\Pbold_k)=C_{G}(P_k)$ and 
$\xbold\in N_{\Pbold_{k}}\left(\Pbold_0\subgroupeq \ldots \subgroupeq \Pbold_{k}\right)$. 
Then in fact $\xbold$ stabilizes $P_0\subgroupeq\ldots\subgroupeq P_k$, since both $n$ and $c$ do so.
By Lemma~\ref{lemma: outer transporters}, the automorphism of $P_k$ induced by $\xbold$ can be induced by some element $y\in P_{k}$, and then $\xbold \cdot y^{-1}\in C_{\Pbold_k}(P_{k})$. Further, since 
$\xbold$ and $\xbold \cdot y^{-1}$ both stabilize
$P_0\subgroupeq\ldots\subgroupeq P_k$, so does~$y$. 
We have expressed
$n=c\cdot \xbold=(c\cdot \xbold\cdot y^{-1})\cdot y$ as an element in the denominator of the left side of~\eqref{eq: map of outs}, which completes the proof.
\end{proof}

\begin{corollary} \label{corollary: in and out}
Let $P$ be a \pdash discretization of a \pdash toral group~$\Pbold$, and let $Q_0\subgroupeq \ldots \subgroupeq Q_k$ be 
a chain of snug subgroups of~$P$. Then 
inclusions of normalizers induce isomorphisms
\[\xymatrix{
\Out_{P}\left(Q_0\subgroupeq\ldots\subgroupeq Q_k\right)\ar[rr]\ar[dr]
&& 
\Out_{\Pbold}\left(\Qbold_0\subgroupeq\ldots\subgroupeq \Qbold_k\right)\\
&\Out_{\Pbold}\left(Q_0\subgroupeq\ldots\subgroupeq Q_k\right)
\ar[ur]^{\cong}_{\rm{Prop.~\ref{proposition: iso of outer autos discrete to cont}}}
}
\]
\end{corollary}

\begin{proof}
The left diagonal map is clearly a monomorphism. 
To see that it is also an epimorphism, 
note that any automorphism of $Q_{k}$ induced 
by conjugation in $\Pbold$ can also be induced by conjugation in~$P$ (Lemma~\ref{lemma: outer transporters}).
\end{proof}

Outer automorphism groups of chains turn out to be
finite, generalizing the result 
of \cite[Prop.~9.4(b)]{BLO-Discrete}
for a single group. 

\begin{lemma}\label{lemma: out of chain is finite}
Let $G$ be a compact Lie group and let 
$P_0\subgroupeq \ldots \subgroupeq P_k$ be a chain of snug discrete \pdash toral subgroups. Then 
$\Out_G\left(P_0\subgroupeq \ldots \subgroupeq P_k\right)$ is finite. 
\end{lemma}

\begin{proof}
We induct on~$k$. The base case is given by \cite[Prop.~9.4(b)]{BLO-Discrete}. Now suppose that
$\Out_G\left(P_0\subgroupeq \ldots \subgroupeq P_{k-1}\right)$ is finite, and consider the homomorphism
induced by deleting the smallest element of the chain: 
\begin{equation} \label{eq: morphism of outs}
\frac{N_G\left(P_0\subgroupeq \ldots \subgroupeq P_{k}\right)}{C_G(P_k)\cdot N_{P_k}\left(P_0\subgroupeq \ldots \subgroupeq P_{k}\right)}
\,\xlongrightarrow{\quad}\,
\frac{N_G\left(P_1\subgroupeq \ldots \subgroupeq P_{k}\right)}{C_G(P_k)\cdot N_{P_k}\left(P_1\subgroupeq \ldots \subgroupeq P_{k}\right)}.
\end{equation}
The target is finite by the inductive hypothesis, so we need only show that the map is injective. The homomorphism of the numerators is certainly injective. Now suppose that 
\[
n\in N_G\left(P_0\subgroupeq \ldots \subgroupeq P_{k}\right)\  \cap\  \left[\strut C_G(P_k)\cdot N_{P_k}\left(P_1\subgroupeq \ldots \subgroupeq P_{k}\right)\right].
\]
Then $n=c\cdot x$ for some 
$c\in C_G(P_k)$ and $x\in N_{P_k}\left(P_1\subgroupeq \ldots \subgroupeq P_{k}\right)$. The element $x$ 
also normalizes $P_0$, since $n$ and $c$ do. Hence $n=c\cdot x \in C_G(P_k)\cdot N_{P_k}\left(P_0\subgroupeq \ldots \subgroupeq P_{k}\right)$, as desired. 
\end{proof}

We turn now to the comparison of the normalizers and centralizers of continuous and discrete \pdash toral subgroups of~$G$. 
The starting point is Proposition~\ref{proposition: discrete to continuous}, which tells us that if $Q\subgroupeq P$ are \pdash discretizations for \pdash toral groups $\Qbold\subgroupeq \Pbold$, then 
$C_{P}(Q)\rightarrow C_{\Qbold}(\Pbold)$ and
$N_{P}(Q)\rightarrow N_{\Pbold}(\Qbold)$ are \pdash discretizations as well (and therefore 
induce mod~$p$ homology equivalences on classifying spaces by
Lemma~\ref{lem:snuglypcompleted}).

To prove Theorem~\ref{theorem: comparison of chains}, we need to relate normalizers of \pdash toral subgroups of an ambient Lie group~$G$ (which could itself be \pdash toral) with normalizers of their \pdash discretizations. We begin with a special case.

\begin{proposition} \label{proposition: p-toral case of normalizers}
Let $P$ be a \pdash discretization of a \pdash toral group~$\Pbold$, and let 
$Q_0\subgroupeq \ldots\subgroupeq Q_{k}$ be a chain of snugly embedded discrete \pdash toral subgroups of~$P$, with closures 
$\Qbold_0\subgroupeq \ldots\subgroupeq \Qbold_{k}$. Then the inclusion
\[
N_{P}\left(Q_0\subgroupeq \ldots\subgroupeq Q_{k}\right) 
\longrightarrow
N_{\Pbold}\left(\Qbold_0\subgroupeq \ldots\subgroupeq \Qbold_{k}\right)
\]
induces a mod~$p$ homology isomorphism on classifying spaces. 
\end{proposition}

\begin{proof}
We induct on the length of the chain. The case $k=0$ is provided by Proposition~\ref{proposition: discrete to continuous}. 

For the inductive hypothesis, assume that for any \pdash toral group $\Bbold$ and \pdash discretization $B\subgroupeq\Bbold$, the inclusion
\begin{equation}   \label{eq: inductive hypothesis}
N_{B}\left(Q_0\subgroupeq \ldots\subgroupeq Q_{k-1}\right) 
\longrightarrow
N_{\Bbold}\left(\Qbold_0\subgroupeq \ldots\subgroupeq \Qbold_{k-1}\right)
\end{equation}
induces a mod~$p$ homology isomorphism of classifying spaces.
Consider the
relationship of the desired statement for $k$ to the corresponding outer automorphism groups:
\begin{equation} \label{eq: ladder for p-toral normalizers}
\begin{gathered}
\xymatrix{
\Largestrut 
N_{P}\left(Q_0\subgroupeq \ldots\subgroupeq Q_{k}\right)
             \ar@{->>}[r]\ar@{^(->}[d]
        & \Largestrut \Out_{P}\left(Q_0\subgroupeq \ldots\subgroupeq Q_{k}\right)
             \ar[d]^{\cong \rm{\ by\ Cor.\ref{corollary: in and out}}}\\
N_{\Pbold}\left(\Qbold_0\subgroupeq \ldots\subgroupeq \Qbold_{k}\right)
             \ar@{->>}[r]
        & \Out_{\Pbold}\left(\Qbold_0\subgroupeq \ldots\subgroupeq \Qbold_{k}\right).
}
\end{gathered}
\end{equation}
The horizontal maps are epimorphisms, 
and the map of kernels is given by
\begin{equation}   \label{eq: map of kernels}
C_{P}(Q_{k})\cdot 
N_{Q_{k}}\left(Q_0\subgroupeq \ldots\subgroupeq Q_{k}\right)
\xrightarrow{\quad}
C_{\Pbold}(\Qbold_{k})\cdot 
N_{\Qbold_{k}}\left(
     \Qbold_0\subgroupeq\ldots\subgroupeq\Qbold_{k}
            \right).
\end{equation}

To streamline notation, let
$Q_{*}\definedas\left(Q_0\subgroupeq \ldots\subgroupeq Q_{k}\right)$ 
and 
$\Qbold_{*}\definedas\left(
     \Qbold_0\subgroupeq\ldots\subgroupeq\Qbold_{k}
            \right)$.
Then \eqref{eq: ladder for p-toral normalizers} 
induces a commutative ladder of classifying spaces, where the rows are fibrations: 
\begin{equation}    \label{eq: the ladder we want}
\begin{gathered}
\xymatrix{
\Largestrut 
B\left(\strut C_{P}(Q_{k})\cdot N_{Q_{k}}\left(Q_*\right) \right)
      \ar[r]\ar[d] 
  & B\left(\strut N_{P}\left(Q_*\right)\right)
    \ar[r]\ar[d]
  & \Largestrut B\Out_{P}\left(Q_*\right)
    \ar[d]^{\cong} 
    \\
B\left(\strut C_{\Pbold}(\Qbold_{k})\cdot N_{\Qbold_{k}}\left(\Qbold_*\right)\right)
    \ar[r] 
  & B\left(N_{\Pbold}\left(\Qbold_*\right)\right)
    \ar[r]
  & B\Out_{\Pbold}\left(\Qbold_*\right).
}
\end{gathered}
\end{equation}
Since the base spaces are the same, it is sufficient to prove that the map between fibers is a mod~$p$ homology isomorphism;
a Serre spectral sequence argument then establishes that the middle map is also an isomorphism on mod~$p$ homology.

To understand the map in \eqref{eq: map of kernels}, 
we need to understand the map on each factor, and also on their intersection. 
The groups $C_{P}(Q_{k})$ and 
  $N_{Q_{k}}\left(Q_*\right)$ are commuting subgroups of~$P$, and their intersection is~$Z(Q_k)$.
We have a central extension
\[
0\longrightarrow 
Z(Q_k)
\longrightarrow 
C_{P}(Q_{k})\times N_{Q_{k}}\left(Q_*\right)
\longrightarrow 
C_{P}(Q_{k})\cdot N_{Q_{k}}\left(Q_*\right)
\longrightarrow 
0
\]
and the analogous one involving $\Qbold_*$ and~$\Pbold$. 
The fibrations induced by these short exact sequences are principal,
and we have the following commutative diagram of horizontal fibrations: 
\begin{equation}    \label{eq: product fibration}
\begin{gathered}
\xymatrix{
BC_{P}(Q_{k})\times BN_{Q_{k}}\left(Q_*\right)
             \ar[r]\ar[d]
        &  B\left(C_{P}(Q_{k})\cdot 
N_{Q_{k}}\left(Q_*\right)\right)
             \ar[d]\ar[r] &   B^2Z(Q_k) \ar[d]\\
BC_{\Pbold}(\Qbold_{k})\times BN_{\Qbold_{k}}\left(\Qbold_*\right)
             \ar[r]
        &  B\left(C_{\Pbold}(\Qbold_{k})\cdot 
N_{\Qbold_{k}}\left(\Qbold_*\right)\right)
             \ar[r] &   B^2Z(\Qbold_k).
}
\end{gathered}
\end{equation}

First consider the fibers. The map on the first factor is a mod~$p$ homology isomorphism by
Proposition~\ref{proposition: discrete to continuous}.
For the second factor, we can apply the inductive hypothesis, because $N_{Q_{k}}\left(Q_*\right)$ is actually $N_{Q_{k}}\left(Q_0\subgroupeq \ldots \subgroupeq Q_{k-1}\right)$ (the normalizer of a shorter chain); likewise $N_{\Qbold_{k}}\left(\Qbold_*\right)$ is $N_{\Qbold_{k}}\left(\Qbold_0\subgroupeq \ldots \subgroupeq \Qbold_{k-1}\right)$.
Therefore 
$BN_{Q_{k}}\left(Q_*\right)\rightarrow BN_{\Qbold_{k}}\left(\Qbold_*\right)$ is a mod~$p$ homology isomorphism. 

Turning to the base, we know that 
$BZ(Q_k)\rightarrow BZ(\Qbold_k)$
induces a mod~$p$ homology isomorphism by 
Corollary~\ref{cor: center is snug}.
The Rothenberg-Steenrod spectral sequence
\cite[Corollary~7.29]{McCleary} then shows that 
$B^2Z(Q_k)\rightarrow B^2Z(\Qbold_k)$ likewise induces an isomorphism on mod~$p$ homology. 

We apply the Serre spectral sequence to~\eqref{eq: product fibration}. The base spaces are simply connected. The maps between the bases and the fibers are mod~$p$ homology isomorphisms. Hence the map of total spaces is a mod~$p$ homology isomorphism as well. Feeding this result back into
\eqref{eq: the ladder we want}
finishes the proof. 
\qedhere
\end{proof}

Finally, we arrive at the proof of this section's main result. 

\begin{proof}[Proof of Theorem~\ref{theorem: comparison of chains}]
Let $P_{*}$ denote the chain $P_0\subgroupeq \ldots\subgroupeq P_{k}$,
and similarly let $\Pbold_{*}$ denote the chain $\Pbold_0\subgroupeq \ldots\subgroupeq\Pbold_{k}$.
We compare the normalizers via the ladder of short exact sequences
\begin{equation} \label{eq: ladder for normalizers}
\begin{gathered}
\xymatrix{
0\ar[r] 
& C_{G}(P_k)\cdot \Largestrut N_{P_{k}}\left(P_*\right)
\ar[r]\ar@{^(->}[d] 
        & \Largestrut 
        N_{G}\left(P_*\right)\ar[r]\ar@{^(->}[d]
        & \Largestrut \Out_{G}(P_*)\ar[r]\ar[d]^{\cong\  \rm{by \ Prop.~\ref{proposition: iso of outer autos discrete to cont}}}
        & 0\\
0\ar[r] & 
C_{G}(\Pbold_{k})\cdot 
N_{\Pbold_{k}}\left(\Pbold_*\right) \ar[r]
        & N_{G}\left(\Pbold_*\right)\ar[r]
        & \Out_{G}(\Pbold_*)\ar[r]
        & 0. 
}
\end{gathered}
\end{equation}
We are in the exact same situation as in the proof of Proposition~\ref{proposition: p-toral case of normalizers}. 
The inclusion $N_{P_k}\left(P_*\right)\hookrightarrow N_{\Pbold_k}\left({\Pbold_*}\right)$ 
induces an isomorphism on mod $p$~homology of classifying spaces by Proposition~\ref{proposition: p-toral case of normalizers}.
The centralizers $C_{G}(P_k)$ and $C_{G}(\Pbold_{k})$ are equal. 
And lastly, $C_{G}(P_k)\cap  N_{P_{k}}\left(P_*\right) =Z(P_k)$
and $C_{G}(\Pbold_{k})\cap 
N_{\Pbold_{k}}\left(\Pbold_*\right)=Z(\Pbold_k)$, and 
$Z(P_k)\rightarrow Z(\Pbold_k)$ induces an isomorphism on mod~$p$ homology of classifying spaces
by Corollary~\ref{cor: center is snug}.
\end{proof}

\bibliographystyle{amsalpha}
\bibliography{WIT-fusion-bibliography}

\end{document}